\documentclass{amsart}  
\usepackage{amssymb,amsthm,amsmath} 
\theoremstyle{plain}
\newtheorem{theorem}{Theorem}[section]

\newtheorem{lemma}[theorem]{Lemma}

\theoremstyle{definition}

\newtheorem{remark}[theorem]{Remark}

\begin{document} 
\title{Graded PI-exponents of  simple \\ Lie superalgebras}
\author[D.D. Repov\v{s}]{Du\v{s}an D. Repov\v{s}}
\address[D.D. Repov\v{s}]{Faculty of Education, and Faculty of Mathematics and Physics,
University of Ljublja\-na
\&
Institute of Mathematics, Physics and Mechanics, 1000 Ljubljana, Slovenia}\email{dusan.repovs@guest.arnes.si}

\author[M.V. Zaicev]{Mikhail V. Zaicev}
\address[M.V. Zaicev]{ 
Department of Algebra,
Faculty of Mathematics and
Mechanics,  Moscow State University,
119992 Moscow, Russia}
\email{zaicevmv@mail.ru}


\thanks{
This research was supported by the Slovenian Research Agency
grants P1-0292, J1-5435, and J1-6721, and the Russian Foundation for Basic Research grant 13-01-00234a. }

\maketitle

\begin{abstract}
We study $\mathbb{Z}_2$-graded identities of simple Lie superalgebras over a field 
of characteristic zero. We prove the existence of the graded PI-exponent for such algebras.
\end{abstract}

\section{Introduction}\label{intro}

Let $A$ be an algebra over a field $F$ with $\rm{char}~F=0$. A natural way of measuring
the polynomial identities satisfied by $A$ is by studying the asymptotic 
behaviour of its sequence of codimensions $\{c_n(A)\}, n=1,2,\ldots~~$. If $A$ is a
finite dimensional algebra then the sequence $\{c_n(A)\}$ is exponentially bounded.
In this case it is natural to ask the question about existence of the limit
\begin{equation}\label{e1}
\lim_{n\to\infty} \sqrt[n]{c_n(A)}
\end{equation}
called the PI-exponent of $A$. 
Such question was first asked for associative algebras by
 Amitsur at the end of 1980's. 
 A positive answer was given in \cite{GZ1}. Subsequently 
 it was shown that the same problem has a positive solution for finite dimensional
Lie algebras \cite{Z}, for finite dimensional alternative and Jordan algebras \cite{GSZ} and for some other classes. Recently  it was shown that in general  the limit (\ref{e1}) does not exist even if $\{c_n(A)\}$ is exponentially bounded \cite{Z-ERA}. 
The counterexample constructed 
in \cite{Z-ERA} is infinite dimensional whereas for finite dimensional
algebras the problem of the existence of the PI-exponent is still open. Nevertheless, if
$\dim A <\infty$ and $A$ is simple then the PI-exponent of $A$ exists as it was proved in 
\cite{GZ2}.

If in addition,  $A$ has a group grading then graded identities, graded codimensions and graded
PI-exponents can also be  considered. In this paper we discuss graded codimensions behaviour 
for finite dimensional simple Lie superalgebras. Graded codimensions  of finite dimensional
Lie superalgebras were studied in a number of papers (see for example, \cite{RZ1} and 
\cite{RZ2}). In particular, in \cite{RZ1} an upper bound of graded codimension growth was found
for one of the series of simple Lie superalgebras.

In the present paper we prove that the graded PI-exponent of any finite dimensional simple Lie
superalgebra always exists. All details concerning numerical PI-theory can be found in 
\cite{GZbook}.

\section{Main constructions and definitions}\label{sec:2}

Let $L=L_0\oplus L_1$ be a Lie superalgebra. Elements from the component $L_0$ are called
{\it even} and elements from $L_1$ are called {\it odd}. Denote by $\mathcal{L}(X,Y)$ a free Lie superalgebra with infinite sets of even generators $X$ and odd generators $Y$. A polynomial $f=f(x_1,\ldots, x_m,y_1,\ldots,y_n) \in \mathcal{L}(X,Y)$ is said to be a {\it graded identity} of Lie superalgebra $L=L_0\oplus L_1$ if $f(a_1,\ldots, a_m,b_1,\ldots,b_n)=0$ whenever $a_1,\ldots, a_m\in L_0, b_1,\ldots,b_n\in L_1$.
 
Denote by
$Id^{gr}(L)$ the set of all graded identities of $L$. Then $Id^{gr}(L)$ is  an ideal of 
$\mathcal{L}(X,Y)$.  Given non-negative integers $0\le k \le n$, let
$P_{k,n-k}$ be the subspace of all multilinear polynomials 
$f=f(x_1,\ldots, x_k,y_1,\ldots,y_{n-k})\in \mathcal{L}(X,Y)$ of degree $k$ on even 
variables and of degree $n-k$ on odd variables. Then $P_{k,n-k}\cap Id^{gr}(L)$ is the
subspace of all multilinear graded identities of $L$ of total degree $n$ depending on $k$ 
even variables and $n-k$ odd variables. Denote also by $P_{k,n-k}(L)$ the quotient
$$
P_{k,n-k}(L)=\frac{P_{k,n-k}}{P_{k,n-k}\cap Id^{gr}(L)}.
$$
Then the {\it partial} graded $(k,n-k)$-codimension of $L$ is
$$
c_{k,n-k}(L)=\dim P_{k,n-k}(L)
$$
and the {\it total} graded $n$th codimension of $L$ is
\begin{equation}\label{e3a}
c_n^{gr}(L)=\sum_{k=0}^n{n\choose k} c_{k,n-k}(L).
\end{equation}

If the sequence $\{c_n^{gr}(L)\}_{n\ge 1}$ is  exponentially bounded then one can
consider the related bounded sequence $\sqrt[n]{c_n^{gr}(L)}$. The latter sequence
has the following lower and upper limits 
$$
\underline{exp}^{gr}(L)=\liminf_{n\to\infty} \sqrt[n]{c_n^{gr}(L)},\qquad
\overline{exp}^{gr}(L)=\limsup_{n\to\infty} \sqrt[n]{c_n^{gr}(L)}
$$
called the {\it lower}  and {\it upper} PI-exponents of $L$, respectively. If the ordinary
limit exists, it is called the (ordinary) {\it graded} PI-{\it exponent} of $L$,
$$
exp^{gr}(L)=\lim_{n\to\infty} \sqrt[n]{c_n^{gr}(L)}.
$$

Symmetric groups and their representations play an important role in the theory of 
codimensions. In particular, in the case of graded identities one can consider the 
$S_k\times S_{n-k}$-action on multilinear graded polynomials. Namely, the subspace
$P_{k,n-k}\subseteq\mathcal{L}(X,Y)$ has a natural structure of 
$S_k\times S_{n-k}$-module where $S_k$ acts on even variables $x_1,\ldots,x_k$ 
while $S_{n-k}$ acts
on odd variables $y_1,\ldots,y_{n-k}$. Clearly, $P_{k,n-k}\cap Id^{gr}(L)$ is the
submodule under this action and we get an induced $S_k\times S_{n-k}$-action on 
$P_{k,n-k}(L)$. The character $\chi_{k,n-k}(L)=\chi(P_{k,n-k}(L))$ is called 
$(k,n-k)$ {\it cocharacter} of $L$. Since $\rm{char}~F=0$, this character can be decomposed 
into the sum of irreducible characters
\begin{equation}\label{e2}
\chi_{k,n-k}(L)=\sum_{{\lambda\vdash k\atop \mu\vdash n-k}}
m_{\lambda,\mu}\chi_{\lambda,\mu}
\end{equation}
where $\lambda$ and $\mu$ are partitions of $k$ and $n-k$, respectively. All details
concerning representations of symmetric groups  can be found in \cite{JK}. An 
application of $S_n$-representations in PI-theory can be found in 
\cite{Baht,
Dr,
GZbook}.
 
Recall that an irreducible $S_k\times S_{n-k}$-module with the character 
$\chi_{\lambda,\mu}$ is the tensor product of $S_k$-module with the character
$\chi_\lambda$ and $S_{n-k}$-module with the character $\chi_\mu$. In particular,
the dimension $\deg\chi_{\lambda,\mu}$ of this module is the product 
$d_\lambda d_\mu$ where $d_\lambda=\deg\chi_\lambda, d_\mu=\deg\chi_\mu$.
Taking into account multiplicities $m_{\lambda,\mu}$ in (\ref{e2}) we get the relation
\begin{equation}\label{e3}
c_{k,n-k}(L)=\sum_{{\lambda\vdash k\atop \mu\vdash n-k}}
m_{\lambda,\mu}d_\lambda d_\mu.
\end{equation}
A number of irreducible components in the decomposition of $\chi_{k,n-k}(L)$, i.e. 
the sum
$$
l_{k,n-k}(L)=\sum_{{\lambda\vdash k\atop \mu\vdash n-k}} m_{\lambda,\mu}
$$ 
is called the $(k,n-k)$-{\it colength} of $L$. The {\it total} graded colength $l_n^{gr}(L)$ is
$$
l_n^{gr}(L)=\sum_{k=0}^n l_{k,n-k}(L).
$$

Now let $L$ be a finite dimensional Lie superalgebra, $\dim L=d$. Then
\begin{equation}\label{e4}
c_n^{gr}(L)\le d^n
\end{equation}
by the results of \cite{BD} (see also \cite{GR}). On the other hand, there exists a polynomial $\varphi$ such that
\begin{equation}\label{e5}
\l_n^{gr} \le \varphi(n)
\end{equation}
for all $n=1,2,\ldots$ as it was mentioned in \cite{RZ1}. Note also that $m_{\lambda,\mu}\ne 0$
in (\ref{e2}) only if $\lambda\vdash k, \mu\vdash n-k$ are partitions with at most $d$
components, that is $\lambda=(\lambda_1,\ldots,\lambda_p),\mu=(\mu_1,\ldots,\mu_q)$ and
$p,q\le d=\dim L$.

Since all partitions under our consideration are of the height at most $d$, we will use the
following agreement. If say, $\lambda$ is a partition of $k$ with $p<d$ components then
we will write $\lambda=(\lambda_1,\ldots,\lambda_d)$ anyway, assuming that
$\lambda_{p+1}=\cdots=\lambda_d=0$.

For studying asymptotic behaviour of codimensions it is convenient to use the following
function defined on partitions. Let $\nu$ be a partition of $m$, $\nu=(\nu_1,\ldots,\nu_d)$.
We introduce the following function of $\nu$:
$$
\Phi(\nu)=\frac{1}{\left(\frac{\nu_1}{m}\right)^\frac{\nu_1}{m}\cdots 
\left(\frac{\nu_{d}}{m}\right)^\frac{\nu_{d}}{m}}.
$$
The values $\Phi(\nu)^m$ and $d_\nu=\deg\chi_\nu$ are very close in the following sense.
\begin{lemma}\label{l1}\rm{(see} \cite[Lemma 1]{GZ2})
Let $m\ge 100$. Then
$$
\frac{\Phi(\nu)^m}{m^{d^2+d}} \le d_\nu \le m \Phi(\nu)^m.
$$
\end{lemma}
\hfill $\Box$

Function $\Phi$ has also the following useful property. 
Let $\nu$ and $\rho$ be two partitions of $m$ with the corresponding Young diagrams
$D_\nu, D_\rho$. We say that $D_\rho$ is obtained from $D_\nu$ by pushing down one
box if there exist $1\le i<j\le d$ such that $\rho_i=\nu_i-1, \rho_j=\nu_j+1$
and $\rho_t=\nu_t$ for all remaining $1\le t \le d$.

\begin{lemma}\label{l2}\rm{(see} \cite[Lemma 3]{GZ2}, \cite[Lemma 2]{ZR})
Let $D_\rho$ be obtained from $D_\nu$ by pushing down one box. Then 
$\Phi(\rho) \ge \Phi(\nu)$.
\end{lemma}
\hfill $\Box$ 

\section{Existence of graded PI-exponents}\label{sec:3}

Throughout this section let $L=L_0\oplus L_1$ be a finite dimensional simple Lie superalgebra,
$\dim L=d$. Then by (\ref{e4}) its upper graded PI-exponent exists,
$$
a=\overline{exp}^{gr}(L)=\limsup_{n\to\infty}\sqrt[n]{c_n^{gr}(L)}.
$$
Note that the even component $L_0$ of $L$ is not solvable since $L$ is simple
(see \cite[Chapter 3, \textsection 2, Proposition 2]{Sch}).

We shall need the following fact.

\begin{remark}
Let $G$ be a non-solvable finite dimensional Lie algebra over a field $F$ of characteristic 
zero. Then the ordinary PI-exponent of $G$ exists and is an integer not less than $2$.
\end{remark}

\begin{proof} It is known that $c_n(G)$ is either polynomially bounded or it grows exponentially not
slower that $2^n$ (see \cite{M}). The first option is possible only if $G$ is solvable. 
On the other hand $exp(G)$ always exists and is an integer \cite{Z} therefore we are done.
\hfill \end{proof}

By the previous remark $P_{n,0}(L) \gtrsim 2^n$ asymptotically and hence
\begin{equation}\label{e6}
a\ge 2.
\end{equation}

The following lemma is the key technical step in the proof of our main result.

\begin{lemma}\label{l3}
For any $\varepsilon>0$ and  any $\delta>0$ there exists an increasing sequence of 
positive integers $n_0,n_1,\ldots~~$such that
\begin{itemize}
\item [(i)] \ \ 
$\sqrt[n]{c_n^{gr}(L)}>(1-\delta)(a-\varepsilon)$ for all $n=n_q, \  q=1,2,\ldots$,
\item [(ii)] \ \ 
$n_{q+1}-n_q\le n_0+d$.
\end{itemize}
\end{lemma}

\begin{proof} Fix $\varepsilon, \delta>0$.
Since $a$ is an upper limit  there exist infinitely many indices $n_0$ such that
$$
c_{n_0}^{gr}(L)> (a-\varepsilon)^{n_0}.
$$
Fixing one of $n_0$ we can find an integer $0\le k_0\le n_0$ such that
\begin{equation}\label{e7}
{n_0\choose k_0}c_{k_0,n_0-k_0}(L)>\frac{1}{n_0+1}(a-\varepsilon)^{n_0}
>\frac{1}{2n_0}(a-\varepsilon)^{n_0}
\end{equation}
(see (\ref{e3a})). Relation (\ref{e5}) shows that
$$
\sum_{{\lambda\vdash k\atop \mu\vdash n-k}}
m_{\lambda,\mu} \le \varphi(n)
$$
for any $0\le k\le n$
where $m_{\lambda,\mu}$ are taken from (\ref{e2}). Then (\ref{e3}) implies
the existence of partitions $\lambda\vdash k_0, \mu\vdash n_0-k_0$ such that
\begin{equation}\label{e8}
{n_0\choose k_0}d_\lambda d_\mu >\frac{1}{2n_0\varphi(n_0)}(a-\varepsilon)^{n_0}.
\end{equation}
The latter inequality means that there exists a multilinear polynomial
$$
f=f(x_1,\ldots,x_{k_0},y_1,\ldots,y_{n_0-k_0})\in P_{k_0,n_0-k_0}
$$ 
such that
$F[S_{k_0}\times S_{n_0-k_0}]f$ is an irreducible 
$F[S_{k_0}\times S_{n_0-k_0}]$-submodule $P_{k_0,n_0-k_0}$ with the character
$\chi_{\lambda,\mu}$ and $f\not\in Id^{gr}(L)$. In particular, there exist
$a_1,\ldots,a_{k_0}\in L_0,b_1,\ldots,b_{n_0-k_0}\in L_1$ such that
$$
A=f(a_1,\ldots,a_{k_0},b_1,\ldots,b_{n_0-k_0})\ne 0
$$
in $L$. First we will show how to find $n_1,k_1$ which are approximately equal 
to $2n_0,2k_0$,
respectively, satisfying the same inequality as (\ref{e7}).

Since $L$ is simple and $A\ne 0$ the ideal generated by $A$ coincides with $L$.
Clearly, every simple Lie superalgebra is centerless. Hence one can find 
$c_1,\ldots,c_{d_1}\in L_0\cup L_1$ such that
$$
[A,c_1,\ldots,c_{d_1},A]\ne 0
$$
and $d_1\le d-1$. Here we use the left-normed notation $[[a,b],c]=[a,b,c]$ for
nonassociative products. It follows that a polynomial
$$
[f_1,z_1,\ldots,z_{d_1},f_2]=g_2\in P_{2k_0+p,2n_0-2k_0+r},\quad p+r=d_1,
$$
is also a non-identity of $L$ where $z_1,\ldots,z_{d_1}\in X\cup Y$ are even or odd
variables, whereas $f_1$ and $f_2$ are copies of $f$ written on disjoint sets of
indeterminates,
$$
f_1=f(x_1^1,\ldots,x_{k_0}^1,y_1^1,\ldots,y_{n_0-k_0}^1),
$$
$$
f_2=f(x_1^2,\ldots,x_{k_0}^2,y_1^2,\ldots,y_{n_0-k_0}^2).
$$

Consider the $S_{2k_0}\times S_{2n_0-2k_0}$-action on $P_{2k_0+p,2n_0-2k_0+r}$
where $S_{2k_0}$ acts on $x_1^1,\ldots,x^1_{k_0}, x_1^2,\ldots,x^2_{k_0}$ 
and $S_{2n_0-2k_0}$ acts on  $y_1^1,\ldots,y^1_{n_0-k_0}, y_1^2,\ldots,y^2_{n_0-k_0}$. 
Denote by $M$ the $F[S_{2k_0}\times S_{2n_0-2k_0}]$-submodule generated by $g_2$ 
and examine its character. It follows from Richardson-Littlewood rule  that
$$
\chi(M)=\sum_{{\nu\vdash 2k_0\atop \rho\vdash 2n_0-2k_0}}
t_{\nu,\rho}\chi_{\nu,\rho}
$$
where either $\nu=2\lambda=(2\lambda_1,\ldots,2\lambda_d)$ or $\nu$ is obtained from $2\lambda$
by pushing down one or more boxes of $D_{2\lambda}$. Similarly, $\rho$ is either equal to $2\mu$  
or $\rho$ is obtained from $2\mu$ by pushing down one or more boxes of $D_{2\mu}$. Then by 
Lemma \ref{l2} we have
$$
\Phi(\nu)\ge \Phi(2\lambda)=\Phi(\lambda),\quad \Phi(\rho)\ge \Phi(2\mu)=\Phi(\mu).
$$
By Lemma \ref{l1} and (\ref{e8}) we have
\begin{equation}\label{e9}
{n_0\choose k_0} 
(\Phi(\lambda)\Phi(\mu))^{n_0} >\frac{1}{2n_0^3\varphi(n_0)}(a-\varepsilon)^{n_0}.
\end{equation}

Now we present the lower bound for binomial coefficients in terms of function $\Phi$.
Clearly, the pair $(k,n-k)$ is a two-component partition of $n$ if $k\ge n-k$.
Otherwise $(n-k,k)$ is a partition of $n$. Since $x^{-x}y^{-y}=y^{-y}x^{-x}$ for all
$x,y\ge 0, x+y=1,$ we will use the notation $\Phi(\frac{k}{n},\frac{n-k}{n})$ in
both cases $k\ge n-k$ or $n-k\ge k$. Then it easily follows from the Stirling formula  that
$$
\frac{1}{n}\Phi\left(\frac{k}{n},\frac{n-k}{n}\right)^n\le {n\choose k}
\le n\Phi\left(\frac{k}{n},\frac{n-k}{n}\right)^n,
$$
hence
\begin{equation}\label{e10}
{qk_0\choose qn_0)} > 
\frac{1}{qn_0}\Phi\left(\frac{qk_0}{qn_0},\frac{qn_0-qk_0}{qn_0}\right)^{qn_0}
=\frac{1}{qn_0}\Phi\left(\frac{k_0}{n_0},\frac{n_0-k_0}{n_0}\right)^{qn_0}
\end{equation}
for all integers $q\ge 2$ and also
\begin{equation}\label{e11}
\left(\Phi\left(\frac{k_0}{n_0},\frac{n_0-k_0}{n_0}\right)
\Phi(\lambda)\Phi(\mu)\right)^{n_0}
>\frac{1}{2n_0^4\varphi(n_0)}(a-\varepsilon)^{n_0}.
\end{equation}
by virtue of  (\ref{e9}).

Recall that we have constructed earlier a multilinear polynomial 
$g_2=[f_1,z_1,\ldots, z_{d_1},f_2]$ which is not a graded identity of $L$ and $f_1,f_2$ 
are copies of $f$. Applying the same procedure we can construct a non-identity of the type
$$
g_q=[g_{q-1},w_1,\ldots, w_{d_{q-1}},f_q]
$$
of total degree $n_{q-1}=n_{q-2}+n_0+w_1+\cdots+w_{d_{q-1}}$ where $d_{q-1}\le d$
and $f_q$ is again a copy of $f$ for all $q\ge 2$.

As in the case $q=2$ the $F[S_{qk_0}\times S_{qn_0-qk_0}]$-submodule of $P_{k,n-k}(L)$
(where $n=n_{q-1}=qn_0+p', k=k_{q-1}=qk_0+p''$) contains an irreducible summand with 
the character $\chi_{\nu,\rho}$ where $\nu\vdash qk_0,\rho\vdash qn_0-qk_0$,
$\Phi(\nu)\ge \Phi(\lambda), \Phi(\rho)\ge \Phi(\mu)$. Moreover, for $n=n_{q-1}$ we have
$$
c_n^{gr}(L)\ge {qn_0\choose qk_0}d_\nu d_\rho >\frac{1}{n^{2d^2+2d}}
{qn_0\choose qk_0}
\left(\Phi(\lambda)\Phi(\mu)\right)^{qn_0}
$$
$$
>\frac{1}{n^{2d^2+2d+1}}
\left(\Phi\left(\frac{k_0}{n_0},\frac{n_0-k_0}{n_0}\right)\Phi(\lambda)\Phi(\mu)\right)^{qn_0}
$$
by Lemma \ref{l1} and the inequality (\ref{e10}). Now it follows from (\ref{e11}) that
$$
c_n^{gr}(L)>\frac{1}{n^{2d^2+2d+1}}
\frac{1}{(2n_0^4\varphi(n_0))^q}(a-\varepsilon)^{qn_0}.
$$
Note that $qn_0\le n \le qn_0+qd$. Hence $q/n\le 1/n_0$ and 
$$
(a-\varepsilon)^{qn_0} \ge\frac{(a-\varepsilon)^{n}}{a^{qd}}
$$
since $a\ge 2$ (see (\ref{e6})). Therefore
$$
\sqrt[n]{c_n^{gr}(L)}>
\frac{(a-\varepsilon)^{n}}{n^\frac{{2d^2+2d+1}}{n} (2a^dn_0^4\varphi(n_0))^\frac{1}{n_0}}
$$
for all $n=n_{q-1}, q=1,2,\ldots \quad$. Finally note that the initial $n_0$ can be taken to 
be arbitrarily large. Hence we can suppose that
$$
n^{-\frac{{2d^2+2d+1}}{n}} (2a^dn_0^4\varphi(n_0))^{-\frac{1}{n_0}}>1-\delta
$$
for all $n\ge n_0$. Hence the inequality
$$
\sqrt[n]{c_n^{gr}(L)}>(1-\delta)(a-\varepsilon)^n
$$
holds for all $n=n_q, q=0,1,\ldots\quad$. The second inequality $n_{q+1}-n_q\le n_0+d$
follows from the construction of the sequence $n_0,n_1,\ldots$, and we have thus
 completed the proof.
 \end{proof}

Now we are ready to prove the main result of the paper.

\begin{theorem}
Let $L$ be a finite dimensional simple Lie superalgebra over a field of characteristic
zero. Then its graded PI-exponent
$$
exp^{gr}(L)=\lim_{n\to\infty} \sqrt[n]{c_n^{gr}(L)}
$$
exists an is less than or equal to $d=\dim L$.
\end{theorem}

\begin{proof} First note that, given a multilinear polynomial 
$$h=h(x_1,\ldots,x_k,y_1,\ldots,y_{n-k})\in P_{k,n-k},$$ the linear span $M$ of all its values 
in $L$ is a $L_0$-module since
$$
[h,z]=\sum_i h(x_1,\ldots,[x_i,z],\ldots,x_k,y_1,\ldots,y_{n-k})
$$
$$
+\sum_j h(x_1,\ldots,x_k,y_1,\ldots,[y_j,z],\ldots,y_{n-k})
$$
for any $z\in\mathcal{L}(X,Y)_0$. Hence $ML_1\ne 0$ in $L$ and $0\equiv [h,w]$ is not 
an identity of $L$ for odd variable $w$ as soon as $h\not\in Id^{gr}(L)$. It follows that
$$
c_{k,n-k+1}(L) \ge c_{k,n-k}(L)
$$
and then
\begin{equation}\label{e12}
c_n^{gr}(L) \ge c_m^{gr}(L)
\end{equation}
for $n\ge m$.

Fix  arbitrary small $\varepsilon,\delta >0$. By Lemma \ref{l3} there exists an 
increasing sequence $n_q,q=1,2,\ldots~$, such that $c_n^{gr}(L)>((1-\delta)(a-\varepsilon))^n$
for all $n=n_q,q=0,1,\ldots~$, and  $n_{q+1}-n_q\le n_0+d$. Denote $b=(1-\delta)(a-\varepsilon)$
and take an arbitrary $n_q<n<n_{q+1}$. Then $c_n^{gr}(L) \ge c_{n_q}^{gr}(L)$ and
$n-n_q\le n_0+d$. Referring to (\ref{e6}), we may assume that $b>1$. Then 
$b^{n_q} \ge b^n\cdot b^{-(n_0+d)}$ and
$$
c_n^{gr}(L) \ge (b^{1-\frac{n_0+d}{n}})^n
$$
for all $n_q\le n \le n_{q+1}$ and all $q=0,1,\ldots~$, that is for all sufficiently
large  $n$. The latter inequality means that
$$
\liminf_{n\to\infty} \sqrt[n]{c_n^{gr}(L)} \ge (1-\delta)b=(1-\delta)^2(a-\varepsilon).
$$
Since $\varepsilon,\delta$ were chosen to be  arbitrary, we have thus completed the proof 
of the theorem.\hfill 
\end{proof}

\end{document}